\documentclass[12 pt]{amsart}
\usepackage{amssymb,latexsym,amsmath,amscd,amsthm,graphicx, color}
\usepackage[all]{xy}
\usepackage{pgf,tikz}
\usepackage{mathrsfs}
\usepackage{cite}
\usetikzlibrary{arrows}
\definecolor{uuuuuu}{rgb}{0.26666666666666666,0.26666666666666666,0.26666666666666666}
\definecolor{xdxdff}{rgb}{0.49019607843137253,0.49019607843137253,1.}
\definecolor{ffqqqq}{rgb}{1.,0.,0.}

\definecolor{uuuuuu}{rgb}{0.26666666666666666,0.26666666666666666,0.26666666666666666}
\definecolor{qqwuqq}{rgb}{0.,0.39215686274509803,0.}
\definecolor{zzttqq}{rgb}{0.6,0.2,0.}
\definecolor{xdxdff}{rgb}{0.49019607843137253,0.49019607843137253,1.}
\definecolor{qqqqff}{rgb}{0.,0.,1.}
\definecolor{cqcqcq}{rgb}{0.7529411764705882,0.7529411764705882,0.7529411764705882}

\theoremstyle{plain}

\newtheorem{theorem}[subsection]{Theorem}

\newtheorem{lemma}[subsection]{Lemma}

\newtheorem{defi}[subsection]{Definition}
\theoremstyle{definition}
\newtheorem{prop}[subsection]{Proposition}
\newtheorem{cor}[subsection]{Corollary}

\newtheorem{remark}[subsection]{Remark}

\newtheorem{note}[subsection]{Note}

\newcommand{\uu}{\cup}

\newcommand{\ci}{\subseteq}
\newcommand{\sci}{\subset}
\newcommand{\es}{\emptyset}
\newcommand{\set}[1]{\{#1\}}

\newcommand{\ga}{\alpha}
\newcommand{\gb}{\beta}

\newcommand{\gd}{\delta}
\renewcommand{\gg}{\gamma}

\newcommand{\gs}{\sigma}
\newcommand{\gt}{\tau}

\newcommand{\tit}{\textit}

\newcommand{\D}[1]{\mathbb{#1}}

\newcommand{\te}{\text}

\begin{document}

To appear, Journal of Fractal Geometry

\title{Quantization and centroidal Voronoi tessellations for probability measures on dyadic Cantor sets}

\author{ Mrinal Kanti Roychowdhury}
\address{School of Mathematical and Statistical Sciences\\
University of Texas Rio Grande Valley\\
1201 West University Drive\\
Edinburg, TX 78539-2999, USA.}
\email{mrinal.roychowdhury@utrgv.edu}

\subjclass[2010]{60Exx, 28A80, 94A34.}
\keywords{Probability measure, Cantor set, quantization, centroidal Voronoi tessellation}
\thanks{The research of the author was supported by U.S. National Security Agency (NSA) Grant H98230-14-1-0320}

\date{}
\maketitle

\pagestyle{myheadings}\markboth{Mrinal Kanti Roychowdhury}{Quantization and centroidal Voronoi tessellations for probability measures on dyadic Cantor sets}

\begin{abstract} Quantization of a probability distribution is the process of estimating a given probability by a discrete probability that assumes only a finite number of levels in its support. Centroidal Voronoi tessellations (CVT) are Voronoi tessellations of a region such that the generating points of the tessellations are also the centroids of the corresponding Voronoi regions.
In this paper, we investigate the optimal quantization and the centroidal Voronoi tessellations with $n$ generators for a Borel probability measure $P$ on $\mathbb R$ supported by a dyadic Cantor set generated by two self-similar mappings with similarity ratios $r$, where $0<r\leq \frac{5-\sqrt{17}}2$.
\end{abstract}

\section{Introduction}

In the context of communication theory, quantization is the process by which data is reduced to a simpler, more coarse representation which is more compatible with digital processing. Loosely speaking, quantization is the heart of analog to digital conversion. It is an area which has increased in importance in the last few decades due to the burgeoning advances in digital technology.  Quantization for probability distributions refers to the idea of estimating a given probability measure by a discrete probability measure with finite support. We refer to \cite{GG, GN, Z} for surveys on the subject and comprehensive lists of references to the literature, see also \cite{AW, DR, GKL, GL}.  For mathematical treatment of quantization one is referred to Graf-Luschgy's book (see \cite{GL1}). Let $\D R^d$ denote the $d$-dimensional Euclidean space, $\|\cdot\|$ denote the Euclidean norm on $\D R^d$ for any $d\geq 1$, and $n\in \D N$. Then, the $n$th \textit{quantization
error} for a Borel probability measure $P$ on $\D R^d$ is defined by
\begin{equation} \label{eq0} V_n:=V_n(P)=\inf \Big\{\int \min_{a\in\alpha} \|x-a\|^2 dP(x) : \alpha \subset \mathbb R^d, \text{ card}(\alpha) \leq n \Big\},\end{equation}
where the infimum is taken over all subsets $\alpha$ of $\mathbb R^d$ with card$(\alpha)\leq n$. If $\int \| x\|^2 dP(x)<\infty$, then there is some set $\alpha$ for
which the infimum is achieved (see \cite{ GKL, GL, GL1}). Such a set $\ga$ for which the infimum occurs and contains no more than $n$ points is called an \tit{optimal set of $n$-means}, or \tit{optimal set of $n$-quantizers}. If $\ga$ is a finite set, in general, the error $\int \min_{a \in \ga} \|x-a\|^2 dP(x)$ is often referred to as the \tit{cost} or \tit{distortion error} for $\ga$, and is denoted by $V(P; \ga)$. Thus, $V_n:=V_n(P)=\inf\set{V(P; \ga) :\alpha \subset \mathbb R^d, \text{ card}(\alpha) \leq n}$. It is known that for a continuous probability measure $P$ an optimal set of $n$-means always has exactly $n$ elements (see \cite{GL1}).
Given a finite subset $\ga\sci \D R^d$, the Voronoi region generated by $a\in \ga$ is defined by
\[M(a|\ga)=\set{x \in \D R^d : \|x-a\|=\min_{b \in \ga}\|x-b\|}\]
i.e., the Voronoi region generated by $a\in \ga$ is the set of all elements in $\D R^d$ which are closer to $a$ than to any other element in $\ga$, and the set $\set{M(a|\ga) : a \in \ga}$ is called the \tit{Voronoi diagram} or \tit{Voronoi tessellation} of $\D R^d$ with respect to $\ga$. A Borel measurable partition $\set{A_a : a \in \ga}$ of $\D R^d$  is called a \tit{Voronoi partition} of $\D R^d$ with respect to $\ga$ (and $P$) if $P$-almost surely
$A_a \sci M(a|\ga)$
for every $a \in \ga$. Notice  that if $\ga=\set{a_1, a_2, \cdots, a_n}$ is an optimal set of $n$-means for $P$ and $\set{A_1, A_2, \cdots, A_n}$ is a Voronoi partition with respect to $\ga$, then
$V_n=\sum_{i=1}^n \int_{A_i} \|x-a_i\|^2 dP(x).$
\tit{Centroidal Voronoi tessellations} (CVTs) are Voronoi tessellations of a region such that the generating points of the tessellations are also the centroids of the corresponding Voronoi regions. A CVT with $n$ generators, also called a CVT with $n$-means, associated with a probability measure $P$ is called an \tit{optimal centroidal Voronoi tessellation} (OCVT) if the $n$ generators form an optimal set of $n$-means for $P$.
Let us now state the following proposition (see \cite{GG, GL1}).
\begin{prop} \label{prop10}
Let $\alpha$ be an optimal set of $n$-means, $a \in \alpha$, and $M(a|\ga)$ be the Voronoi region generated by $a\in \ga$, i.e.,
$M(a|\ga)=\{x \in \mathbb R^d : \|x-a\|=\min_{b \in \alpha} \|x-b\|\}.$
Then, for every $a \in\alpha$,
$(i)$ $P(M(a|\ga))>0$, $(ii)$ $ P(\partial M(a|\ga))=0$, $(iii)$ $a=E(X : X \in M(a|\ga))$, and $(iv)$ $P$-almost surely the set $\set{M(a|\ga) : a \in \ga}$ forms a Voronoi partition of $\D R^d$.
\end{prop}

\begin{remark} \label{rem1}
Let $\alpha$ be an optimal set of $n$-means and  $a \in \alpha$, then by Proposition~\ref{prop10}, we have
\begin{align*}
a=\frac{1}{P(M(a|\ga))}\int_{M(a|\ga)} x dP=\frac{\int_{M(a|\ga)} x dP}{\int_{M(a|\ga)} dP},
\end{align*}
which implies that $a$ is the centroid of the Voronoi region $M(a|\ga)$ associated with the probability measure $P$ (see also \cite{DFG}). Thus, we can say that for a Borel probability measure $P$ on $\D R^d$, an optimal set of $n$-means forms a centroidal Voronoi tessellation of $\D R^d$; however, the converse is not true in general (see \cite{DFG, GG}).
\end{remark}
Let $S_1, S_2 : \D R \to \D R$ be two contractive similarity mappings such that
$S_1(x)=r x \text{ and } S_2 (x)=r x +(1-r)$
for $0<r<\frac 12$. Then, there exists a unique Borel probability measure $P$
on $\mathbb R$ such that
$P=\frac 12 P\circ S_1^{-1}+\frac 12 P\circ S_2^{-1}$, where $P\circ S_i^{-1}$ denotes the image measure of $P$ with respect to
$S_i$ for $i=1, 2$ (see \cite{H}). Such a $P$ has support the Cantor set generated by the two mappings $S_1$ and $S_2$. In this paper, in Section 3, we have given a centroidal Voronoi tessellation (CVT) for the probability measure $P$ supported by the Cantor set generated by $S_1(x)=\frac 4 9 x$ and $S_2(x)=\frac 4 9 x+\frac 59$. The formula in this paper can be used to obtain a CVT for $P$ on any Cantor set generated by $S_1(x)=rx$ and $S_2(x)=rx+(1-r)$, where $0.4364590141\leq r\leq 0.4512271429$ (written up to ten decimal places). For the classical Cantor set $C$, i.e., when $r=\frac 13$, in the paper \cite{GL2},  Graf and Luschgy determined the optimal sets of $n$-means for the probability measure $P$ for all $n\geq 2$. For a long time it was believed that using the same formula given in \cite{GL2}, one could determine the optimal sets of $n$-means for all $n\geq 2$ for the probability measure $P$ supported by any Cantor set generated by the two mappings $S_1(x)=rx$ and $S_2(x)=rx+(1-r)$ for $0<r\leq \frac {5-\sqrt {17}}{2}$, i.e., if $0<r\leq 0.4384471872$. In Proposition~\ref{prop4} we have shown that it is not always true, by showing that if $0.4371985206<r\leq 0.4384471872$ and $n$ is not of the form $2^{\ell(n)}$ for any positive integer $\ell(n)$, then the distortion error of the CVT obtained using the formula in this paper is smaller than the distortion error of the CVT obtained using the formula given by Graf-Luschgy in \cite{GL2}. In fact, in Section~5, we have further improved this bound which is given in Remark~\ref{remark00}. In addition, the work in this paper shows that under squared error distortion measure, the centroid condition is not sufficient for optimal quantization for a singular continuous probability measure. Recall that the centroid condition is not sufficient for optimal quantization for an absolutely continuous probability measure is already known (see \cite{DFG} and \cite[Chapter~6]{GG}).

\section{Preliminaries}
By a \textit{string} or a \textit{word} $\sigma$ over an alphabet $\{1, 2\}$, we mean a finite sequence $\gs:=\gs_1\gs_2\cdots \gs_k$
of symbols from the alphabet, where $k\geq 1$, and $k$ is called the length of the word $\gs$.  A word of length zero is called the \textit{empty word}, and is denoted by $\emptyset$.  By $\{1, 2\}^*$, we denote the set of all words
over the alphabet $\{1, 2\}$ of some finite length $k$ including the empty word $\emptyset$. By $|\gs|$, we denote the length of a word $\gs \in \{1, 2\}^*$. For any two words $\gs:=\gs_1\gs_2\cdots \gs_k$ and
$\tau:=\tau_1\tau_2\cdots \tau_\ell$ in $\{1, 2\}^*$, by
$\gs\tau:=\gs_1\cdots \gs_k\tau_1\cdots \tau_\ell$, we mean the word obtained from the
concatenation of the two words $\gs$ and $\tau$. Let $S_1$ and $S_2$ be two contractive similarity mappings on $\D R$ given by $S_1(x)=\frac 4 9 x$ and $S_2(x)=\frac 49 x+\frac 59$. For $\gs:=\gs_1\gs_2 \cdots\gs_k \in \{ 1, 2\}^k$, set $S_\gs:=S_{\gs_1}\circ \cdots \circ S_{\gs_k}$
and $J_\gs:=S_{\gs}([0, 1])$. For the empty word $\es$, by $S_\es$, it is meant the identity mapping on $\D R$, and write $J:=J_\es=S_\es([0,1])=[0, 1]$. Then, the set $C:=\bigcap_{k\in \mathbb N} \bigcup_{\gs \in \{1, 2\}^k} J_\gs$ is known as the \textit{Cantor set} generated by the two mappings $S_1$ and $S_2$, and equals the support of the probability measure $P$ given by $P=\frac 12 P\circ S_1^{-1}+\frac 12 P\circ S_2^{-1}$. For any $\gs \in \set{1, 2}^\ast$, the intervals $J_{\gs1}$ and $J_{\gs2}$ into which $J_\gs$ is split up are called the basic intervals of $J_\gs$. For $\gs=\gs_1\gs_2 \cdots\gs_k \in \{ 1, 2\}^\ast$, $k\geq 0$, write
$p_\gs:=\frac 1{2^k}$ and $s_\gs:=(\frac 49)^k$.

Let $X$ be a random variable with probability distribution $P$. By $E(X)$ and $V:=V(X),$ we mean the expectation and the variance of the random variable $X$.
For words $\gb, \gg, \cdots, \gd$ in $\set{1,2}^\ast$, by $a(\gb, \gg, \cdots, \gd)$, we mean the conditional expectation of the random variable $X$ given $J_\gb\uu J_\gg \uu\cdots \uu J_\gd,$ i.e.,
\begin{equation} \label{eq45} a(\gb, \gg, \cdots, \gd)=E(X|X\in J_\gb \uu J_\gg \uu \cdots \uu J_\gd)=\frac{1}{P(J_\gb\uu \cdots \uu J_\gd)}\int_{J_\gb\uu \cdots \uu J_\gd} x dP.
\end{equation}
Let us now give the following lemmas.

\begin{lemma} \label{lemma1}
Let $f : \mathbb R \to \mathbb R^+$ be Borel measurable and $k\in \mathbb N$. Then
\[\int f dP=\sum_{\sigma \in \{1, 2\}^k} \frac 1 {2^k} \int f \circ S_\sigma dP.\]
\end{lemma}

\begin{proof}
We know $P=\frac 1 2  P\circ S_1^{-1} +\frac 12  P\circ S_2^{-1}$, and so by induction $P=\sum_{\sigma \in \{1, 2\}^k} \frac 1 {2^k} P\circ S_\sigma^{-1}$, and thus the lemma is yielded.
\end{proof}

\begin{lemma} \label{lemma2}
$E(X)=\frac 12  \text{ and } V:=V(X)=\frac {5}{52},$ and for any $x_0 \in \mathbb R$,
$\int (x-x_0)^2 dP(x) =V (X) +(x_0-\frac 12)^2.$
\end{lemma}
\begin{proof}
 We have
$E(X)=\int x dP(x)=\frac 1 2\int \frac 49 x dP(x) +\frac 12  \int (\frac 49 x +\frac 59 ) dP(x)= \frac 4{18} \,E(X) +\frac 4{18} \, E(X) +\frac 5{18} =\frac  4 {9}\, E(X)+ \frac 5 {18}$,
which implies $E(X)=\frac 1 2.$
\begin{align*}
&E(X^2)=\int x^2 dP(x)=\frac 1 2 \int x^2 dP\circ S_1^{-1}(x)+\frac 12 \int x^2 dP\circ S_2^{-1}(x) \\
&=\frac 1 2  \int \frac {16} {81}\, x^2 dP(x) +\frac 12   \int \Big(\frac 4 9 \, x +\frac 5 9 \Big)^2 dP(x)= \frac {16}{81}\, E(X^2)  +\frac {20}{81} \, E(X) +\frac {25} {162}\\
&=\frac {16}{81}\,E(X^2) +\frac {20}{162} +\frac {25} {162},
\end{align*}
which implies $E(X^2)=\frac {9}{26}$, and hence
$V(X)=E(X-E(X))^2=E(X^2)-\left(E(X)\right)^2 =\frac {9}{26}-(\frac 12 )^2=\frac {5}{52}$. Then, following the standard theory of probability, we have
$ \int(x-x_0)^2 dP(x) =V(X)+(x_0-E(X))^2,$ and thus the lemma is yielded.
\end{proof}

\begin{cor} \label{cor1}
Let $\gs \in \set{1, 2}^\ast$. Then, for any $x_0 \in \mathbb R$,
\begin{equation} \label{eq234} \int_{J_\gs} (x-x_0)^2 dP(x) =p_\gs\Big(s_\gs^2 V  +(S_\gs(\frac 12)-x_0)^2\Big).\end{equation}
\end{cor}

\begin{note}  \label{note1}
Notice  that from the above lemma it follows that the optimal set of one-mean is the expected value and the corresponding quantization error is the variance $V$ of the random variable $X$.
For $\sigma \in \{1, 2\}^k$, $k\geq 1$, since $a(\gs)=E(X : X \in J_\sigma)$,  using Lemma~\ref{lemma1}, we have
\begin{align*}
&a(\gs)=\frac{1}{P(J_\sigma)} \int_{J_\sigma} x \,dP(x)=\int_{J_\sigma} x\, dP\circ S_\sigma^{-1}(x)=\int S_\sigma(x)\, dP(x)=E(S_\sigma(X)).
\end{align*}
Since $S_1$ and $S_2$ are similarity mappings, it is easy to see that $E(S_j(X))=S_j(E(X))$ for $j=1, 2$ and so by induction, $a(\gs)=E(S_\sigma(X))=S_\sigma(E(X))=S_\sigma(\frac 12)$
 for $\sigma\in \{1, 2\}^k$, $k\geq 1$.
\end{note}
In the next section, Proposition~\ref{prop1}, Proposition~\ref{prop2} and Proposition~\ref{prop3} determine the centroidal Voronoi tessellations with $n$ generators for the probability measure $P$ for all $n\geq 2$.

\section{Centroidal Voronoi tessellations for all $n\geq 2$}

In this section, we determine the CVTs with $n$-means for each $n\geq 2$ of the Cantor set $C$ generated by the two mappings $S_1$ and $S_2$ defined by $S_1(x)=\frac 49 x$ and $S_2(x)=\frac 49 x+\frac 59$ for $x\in \D R$. As the probability distribution $P$ has support the Cantor set $C$ and $C\sci J$, a CVT of $J$ with respect to the probability distribution is also a CVT of $C$ and vice versa. Once we know a CVT, using the formula \eqref{eq234}, the corresponding distortion error can easily be obtained.  Write
$
 A_\gs:=\set{a(\gs11, \gs121, \gs1221), a(\gs1222, \gs21), a(\gs22)}, \te{ or } \\
 A_\gs:=\set{a(\gs11),  a(\gs12, \gs2111), a(\gs2112, \gs212, \gs22)}$. If $\gs$ is the empty word $\es$, then we have $A:=A_\es=\set{a(11, 121, 1221), a(1222, 21), a(22)}$, or \\$A:=A_\es=\set{a(11),  a(12, 2111), a(2112, 212, 22)}$.

Let us now prove the following lemma.
\begin{lemma} \label{lemma41}
Similarity mappings preserve the ratio of the distances of a point from any other two points.
\end{lemma}
\begin{proof}
Let $a, b, c\in \D R$. For any $\gs \in \set{1, 2}^k$ it is enough to prove that $\frac{|b-a|}{|c-b|}=\frac{|S_\gs(b)-S_\gs(a)|}{|S_\gs(c)-S_\gs(b)|}$, which is clearly true, since
\[\frac{|S_\gs(b)-S_\gs(a)|}{|S_\gs(c)-S_\gs(b)|}=\frac{s_\gs|b-a|}{s_\gs|c-b|}=\frac{|b-a|}{|c-b|}.\]
\end{proof}
\begin{remark} \label{rem2} The two mappings $S_1$ and $S_2$ defined in this paper are increasing mappings, i.e., for any $x, y \in \D R$,  $x<y$ implies $S_i(x)<S_i(y)$ for $i=1, 2$, and so by Lemma~\ref{lemma41} if $a<b<c$, we have
\[(c-b)(S_\gs(b)-S_\gs(a))=(b-a)(S_\gs(c)-S_\gs(b)).\]
\end{remark}

\begin{prop}\label{prop1}
Let $n\in \D N$ and $n=2^k$ for some $k\in \D N$. Then,
$\ga_n=\set{S_\gs(\frac 1 2) : \gs \in \set{1, 2}^k}$
forms a unique optimal CVT with $n$-means with distortion error $V_n=\left(\frac 4  9\right)^{2k} V$.
\end{prop}
\begin{proof} By Remark~\ref{rem2}, for any $\gs \in \set{1, 2}^k$ we have
\[(1-\frac 12)(S_\gs(\frac 12)-S_\gs(0))=(\frac 12-0)(S_\gs(1)-S_\gs(\frac 12)),\] which implies
$S_\gs(\frac 12)=\frac 12(S_\gs(0)+S_\gs(1))$, i.e.,  $S_\gs(\frac 12)$ are the midpoints of the basic intervals $J_\gs$ for all $\gs \in \set{1, 2}^k$. In addition, by Remark~\ref{rem1} and Note~\ref{note1}, $S_\gs(\frac 12)$ is the centroid of $J_\gs$. Thus, the set $\set{S_\gs(\frac 12) : \gs \in \set{1, 2}^k}$ forms a CVT of the Cantor set. Moreover, by Corollary~\ref{cor1}, for $a\in \D R$, $\int_{J_\gs}(x-a)^2 dP$ is minimum when $a=S_\gs(\frac 12)$. Hence, the set $\ga_n$ forms a unique optimal CVT of the Cantor set, and then
\[V_n=\int \min_{a \in \ga_n}\|x-a\|^2 dP=\sum_{\gs \in \set{1, 2}^k} \int_{J_\gs}\min_{a \in \ga}(x-a)^2  dP=\sum_{\gs \in \set{1, 2}^k} p_\gs s_\gs^2 V=\left(\frac 49\right)^{2k} V.\]
This completes the proof of the proposition.
\end{proof}
\begin{lemma}\label{lemma43} Let
$A =\set{a(11, 121, 1221), a(1222, 21), a(22)}$  or \\
$A =\set{a(11),  a(12, 2111), a(2112, 212, 22)}$.
Then, $A$ forms a CVT with three-means of the Cantor set $C$.

\begin{proof}
We have
\begin{align*} S_{1221}(1)=0.395671&<\frac{1}{2}(a(11, 121, 1221)+ a(1222, 21))=0.400854<S_{1222}(0)=0.405426,\\
S_{21}(1)=0.753086&<\frac{1}{2}(a(1222, 21)+a(22))=0.754839<S_{22}(0)=0.802469.
\end{align*}
Thus, $A=\set{a(11, 121, 1221), a(1222, 21), a(22)}$ forms a CVT of $C$.  Due to symmetry $A=\set{a(11),  a(12, 2111), a(2112, 212, 22)}$ also forms a CVT of $C$.
\end{proof}

\end{lemma}

\begin{lemma}\label{lemma44}  The set $\set{a(111, 1121, 11221), a(11222, 121), a(122), a(21), a(22)}$ forms a CVT with five-means.
\end{lemma}
\begin{proof} We have $S_1^{-1}(\set{a(111, 1121, 11221), a(11222, 121), a(122)})\\=\set{a(11, 121, 1221), a(1222, 21), a(22)}$. So, by Lemma~\ref{lemma43}, the set \\
$S_1^{-1}(\set{a(111, 1121, 11221), a(11222, 121), a(122)})$ forms a CVT with three-means. Similarly, by Proposition~\ref{prop1}, the set $S_2^{-1}(\set{a(21), a(22)})=\set{a(1), a(2)}$ forms a CVT with two-means. By Lemma~\ref{lemma41}, we know that the similarity mappings preserve the ratio of the distances of a point from any other two points, and so the set $\set{a(111, 1121, 11221), a(11222, 121), a(122)}$ forms a CVT with three-means of $J_1$ and the set $\set{a(21), a(22)}$ forms a CVT with two-means of $J_2$. Thus, the union of the CVTs of $J_1$ and $J_2$ will form a CVT of the Cantor set $C$ if we can prove that
\[S_{122}(1)\leq \frac 12(a(122)+a(21))\leq S_{21}(0),\]
which is clearly true since
\[S_{122}(1)=0.444444< \frac 12(a(122)+a(21))=0.527435< S_{21}(0)=0.555556.\]
Hence the given set forms a CVT with five-means.
\end{proof}

\begin{remark}
Similarly, we can prove that the sets \\$\set{a(111),  a(112, 12111), a(12112, 1212, 122), a(21), a(22)}$,\\ $\set{a(11), a(12), a(211, 2121, 21221), a(21222, 221), a(222)}$, \\ $\set{a(11), a(12), a(211),  a(212, 22111), a(22112, 2212, 222)}$ also form CVTs with five-means, i.e., the number of CVTs with five-means is four.
\end{remark}

\begin{lemma}  \label{lemma45}
The set $\set{a(111, 1121, 11221), a(11222, 121), a(122), a(211, 2121, 21221), \\ a(21222, 221), a(222)}$ forms a CVT with six-means.
\end{lemma}
\begin{proof} The set $\set{a(11, 121, 1221), a(1222, 21), a(22)}$ forms a CVT of $J$ with three-means. So, by Lemma~\ref{lemma41}, the sets
$S_1(\set{a(11, 121, 1221), a(1222, 21), a(22)})$  and \\ $S_2(\set{a(11, 121, 1221), a(1222, 21), a(22)})$ form CVTs of $J_1$ and $J_2$, respectively. Thus, the given set will form a CVT with six-means if we can prove that
\[S_{122}(1)\leq \frac 12(a(122)+ a(211, 2121, 21221))\leq S_{211}(0),\]
which is clearly true since
\[S_{122}(1)=0.444444< \frac 12(a(122)+ a(211, 2121, 21221))=0.521< S_{211}(0)=0.555556.\]
Thus, the lemma is yielded.
\end{proof}
\begin{remark} By Lemma~\ref{lemma43}, since there are two different CVTs of $J$ with three-means, one can say that each of the basic intervals $J_1$ and $J_2$ has two different CVTs, and thus using all possible combinations one can see that the total number of CVTs with six-means is four.
\end{remark}

\begin{lemma}  \label{lemma46}
The set $\set{a(111, 1121, 11221), a(11222, 121), a(122), a(211), a(212), a(221), a(222)}$ forms a CVT with seven-means.
\end{lemma}

\begin{proof}
The set $\set{a(11, 121, 1221), a(1222, 21), a(22)}$ forms a CVT of $J$ with three-means. So, by Lemma~\ref{lemma41}, the set $\set{a(111, 1121, 11221), a(11222, 121), a(122)}$  forms a CVT of $J_1$ with three-means. Again by Proposition~\ref{prop1}, the set $\set{a(211), a(212), a(221), a(222)}$ forms a CVT of $J_2$ with four-means. Hence, the union of the two CVTs will form a CVT with seven-means if we can prove that
\[S_{122}(1)\leq \frac 12(a(122)+ a(211))\leq S_{211}(0),\]
which is clearly true since
\[S_{122}(1)=0.444444< \frac 12(a(122)+ a(211))=0.5< S_{211}(0)=0.555556.\]
Thus, the lemma is obtained.
\end{proof}
\begin{remark} Each $J_i$ for $i=1, 2$, has two different CVTs, and so using all possible combinations we see that the total number of CVTs with seven-means is four.
\end{remark}

Let us now prove the following two propositions.
\begin{prop} \label{prop2}
Let $n \in \D N$ be such that $n=2^{\ell(n)}+k$, where $1\leq k\leq  2^{\ell(n)-1}$. Let $I \ci \set{1, 2}^{\ell(n)-1}$ with card$(I)=k=n-2^{\ell(n)}$.  Then, the set
$\ga_n:=\ga_n(I)$, where
\[\ga_n(I)=\left\{\begin{array}{ll}
\uu_{\gs \in I} A_\gs \uu\Big\{S_{\gs1}(\frac 1 2 ), S_{\gs2}(\frac 1 2) : \gs \in \set{1, 2}^{\ell(n)-1}\setminus I\Big\} \te{ if } 1\leq k< 2^{\ell(n)-1}, & \\
\uu_{\gs \in I} A_\gs  \te{ if } k= 2^{\ell(n)-1}\neq 1, &\\
A_\es \te{ if } k= 2^{\ell(n)-1}=1, \te{ i.e., when } n=3, &
\end{array}\right.\]
forms a CVT with $n$-means.
The number of CVTs for $1\leq k< 2^{\ell(n)-1}$ is  $2^{n-2^{\ell(n)}} \times {}^{2^{\ell(n)-1}}C_{n-2^{\ell(n)}}$, and the number of CVTs for $k= 2^{\ell(n)-1}$ is $2^{2^{\ell(n)-1}}$.
\end{prop}

\begin{proof} Let us first assume that $n=2^{\ell(n)}+k$, where $1\leq k< 2^{\ell(n)-1}$. Let $I \ci \set{1, 2}^{\ell(n)-1}$ with card$(I)=n-2^{\ell(n)}$. By Lemma~\ref{lemma43}, for each $\gs \in I$, the set $S_\gs^{-1}(A_\gs)$ forms a CVT of $J$ with three-means, and so by Lemma~\ref{lemma41} the set $A_\gs$ forms a CVT of $J_\gs$ with three-means. Now, proceeding in the similar way as Lemma~\ref{lemma45}, we can say that the set $\uu_{\gs \in I} A_\gs$ forms a CVT of $\uu_{\gs \in I} J_\gs$. Again, for each $\gs \in \set{1, 2}^{\ell(n)-1}\setminus I$ the set $S_\gs^{-1}(\set{S_{\gs1}(\frac 1 2 ), S_{\gs2}(\frac 1 2)})$ forms a CVT of $J$ with two-means, and so by Lemma~\ref{lemma41} the set $\set{S_{\gs1}(\frac 1 2 ), S_{\gs2}(\frac 1 2)}$ forms a CVT of $J_\gs$ with two-means. Now, proceeding in the similar way as Lemma~\ref{lemma44}, we can say that the set $\uu_{\gs \in I} A_\gs \uu\Big\{S_{\gs1}(\frac 1 2 ), S_{\gs2}(\frac 1 2) : \gs \in \set{1, 2}^{\ell(n)-1}\setminus I\Big\}$ forms a CVT with $n$-means. Notice  that $\te{card}(I)=n-2^{\ell(n)}$ and $I$ can be chosen in $ {}^{2^{\ell(n)-1}}C_{n-2^{\ell(n)}}$ ways. For each $\gs \in I$ there are two different choices for $A_\gs$. Hence, the number of CVTs in this case is $2^{n-2^{\ell(n)}} \times {}^{2^{\ell(n)-1}}C_{n-2^{\ell(n)}}$.
Similarly, by Lemma~\ref{lemma41}, Lemma~\ref{lemma43} and proceeding in the similar way as Lemma~\ref{lemma45}, we can prove that if $n=2^{\ell(n)}+k$, where $k= 2^{\ell(n)-1}\neq 1$ or $k= 2^{\ell(n)-1}= 1$, the set $\uu_{\gs \in I} A_\gs$ forms a CVT of $C$, and the number of CVTs in either case is given by $2^{2^{\ell(n)-1}}$. Hence, the proposition is yielded.
\end{proof}

\begin{prop} \label{prop3}
Let $n \in \D N$ be such that $n=3\cdot 2^{\ell(n)-1}+k$, where $1\leq k\leq 2^{\ell(n)-1}-1 $. Let $I \ci \set{1, 2}^{\ell(n)-1}$ with card$(I)=n-3 \cdot 2^{\ell(n)-1}$.  Then, the set
$\ga_n:=\ga_n(I)$, where
\[
\ga_n(I)=(\uu_{\gs \in \set{1, 2}^{\ell(n)-1}\setminus I} A_\gs)\uu \Big\{S_{\gs \gt}\Big(\frac 1 2\Big ):   \gs \in I \te{ and } \gt \in \set{ 1, 2}^2\Big\},
\]
forms a CVT with $n$-means.
The number of such sets is  $2^{2^{\ell(n)+1}-n}\times {}^{2^{\ell(n)-1}}C_{n-3 \cdot 2^{\ell(n)-1}}$.
\end{prop}

\begin{proof} Let $n=3\cdot 2^{\ell(n)-1}+k$, where $1\leq k\leq 2^{\ell(n)-1}-1$. Let $I \ci \set{1, 2}^{\ell(n)-1}$ with card$(I)=n-3 \cdot 2^{\ell(n)-1}$. For each $\gs \in \set{1, 2}^{\ell(n)-1}\setminus I$ we have $S_\gs^{-1}(A_\gs)=A$, which is a CVT of $C$ with three-means, and so the set $A_\gs$ forms a CVT of $J_\gs$ for each $\gs \in \set{1, 2}^{\ell(n)-1}\setminus I$. Again, the set $\set{J_{\gs\gt} : \gt \in \set{ 1, 2}^2}$ forms a CVT with four-means of $J_\gs$ for each $\gs \in I$. Thus, proceeding in the similar way as Lemma~\ref{lemma46}, we can prove that the set $\ga_n(I)=(\uu_{\gs \in \set{1, 2}^{\ell(n)-1}\setminus I} A_\gs)\uu \Big\{S_{\gs \gt}\Big(\frac 1 2\Big ):   \gs \in I \te{ and } \gt \in \set{ 1, 2}^2\Big\}$ forms a CVT of $C$ with $n$-means. Notice  that $I$ can be chosen in $  {}^{2^{\ell(n)-1}}C_{n-3 \cdot 2^{\ell(n)-1}}$ ways and $\te{card}(\set{1, 2}^{\ell(n)-1}\setminus I)=2^{\ell(n)-1}-(n-3 \cdot 2^{\ell(n)-1})=2^{\ell(n)+1}-n$. For each $\gs \in \set{1, 2}^{\ell(n)-1}\setminus I$ there are two different choices for $A_\gs$. Hence, the number of CVTs in this case is $2^{2^{\ell(n)+1}-n}\times {}^{2^{\ell(n)-1}}C_{n-3 \cdot 2^{\ell(n)-1}}$. Thus, the proposition is yielded.
\end{proof}

Let us now prove the following lemma.
\begin{lemma} \label{lemma456}
Let $P=\frac 1 2 P\circ S_1^{-1}+\frac 12 P\circ S_2^{-1}$ be the probability measure supported by the Cantor set generated by $S_1(x)=rx$ and $S_2(x)=rx+(1-r)$. Let $n\in \D N$, $n\geq 2$ and $n$ is not of the form $2^{\ell(n)}$ for any $\ell(n)\in \D N$. Then, $\ga_n(I)$, given by Proposition~\ref{prop2} or Proposition~\ref{prop3}, for each $n\geq 2$ forms a CVT if $0.4364590141\leq r\leq 0.4512271429$ (written up to ten decimal places).
\end{lemma}
\begin{proof} Let $\ga_3(I)=\set{a(11, 121, 1221), a(1222, 21), a(22)}$. It forms a CVT if
\begin{align*} S_{1221}(1)&\leq \frac 12(a(11, 121, 1221)+ a(1222, 21))\leq S_{1222}(0)  \\
\te { and } & S_{21}(1)\leq \frac 12 (a(1222, 21)+a(22))\leq S_{22}(0),
 \end{align*} i.e., if $0.4364590141\leq r\leq 0.4521904271$ and $0.2076973455\leq r\leq 0.4512271429$, which yields  $0.4364590141\leq r\leq 0.4512271429$. Similarly, if
 $\ga_3(I)=\set{a(11),  a(12, 2111), a(2112, 212, 22)}$, it will form a CVT if $0.4364590141\leq r\leq 0.4512271429$. By Lemma~\ref{lemma41}, we can say that the set $\ga_n(I)$ for each $n\geq 2$ also forms a CVT if $0.4364590141\leq r\leq 0.4512271429$, and thus the lemma is yielded.

\end{proof}

\begin{remark} Proposition~\ref{prop1}, Proposition~\ref{prop2} and Proposition~\ref{prop3} give the CVTs with $n$-means for the probability distribution $P$ supported by the Cantor set generated by the mappings $S_1(x)=\frac 4 9 x$ and $S_2(x)=\frac 4 9 x+\frac 59$ for any positive integer $n\geq 2$. Lemma~\ref{lemma456} says that using the formula given in this paper, if $n$ is not of the form $2^{\ell(n)}$ for any $\ell(n)\in \D N$, one can determine the CVTs with $n$-means for any $n\geq 2$, and hence the corresponding distortion error, for the probability measure $P$ supported by any Cantor set generated by $S_1(x)=rx$ and $S_2(x)=rx+(1-r)$, where $0.4364590141\leq r\leq 0.4512271429$.
\end{remark}

\section{Distortion errors for two different CVTs}

In this section we compare the distortion errors for two different CVTs with $n$-means: one is obtained using the formula given by Proposition~\ref{prop1}, Proposition~\ref{prop2}, or Proposition~\ref{prop3} in this paper, and one is obtained using the formula given in \cite{GL2}. Let $P=\frac 1 2 P\circ S_1^{-1}+\frac 12 P\circ S_2^{-1}$ be the probability measure supported by the Cantor set generated by $S_1(x)=rx$ and $S_2(x)=rx+(1-r)$. Then, it can be shown that if $V$ is the variance of a random variable with distribution $P$ in this case, then
$V=\frac{1-r}{4 (r+1)}$.

\begin{defi}  \label{def1} For $n\in \D N$ with $n\geq 2$ let $\ell(n)$ be the unique natural number with $2^{\ell(n)} \leq n<2^{\ell(n)+1}$. For $I\sci \set{1, 2}^{\ell(n)}$ with card$(I)=n-2^{\ell(n)}$ let $\gb_n(I)$ be the set consisting of all midpoints $a_\gs$ of intervals $J_\gs$ with $\gs \in \set{1,2}^{\ell(n)} \setminus I$ and all midpoints $a_{\gs 1}$, $a_{\gs 2}$ of the basic intervals of $J_\gs$ with $\gs \in I$. Formally,
\[\gb_n(I)=\set{a_\gs : \gs \in \set{1,2}^{\ell(n)} \setminus I} \uu \set{a_{\gs 1} : \gs \in I} \uu \set {a_{\gs 2} : \gs \in I}.\]
\end{defi}
In \cite{GL2}, it was shown that $\gb_n(I)$ forms an optimal set of $n$-means for $r=\frac 13$.
Let us now prove the following lemma.

\begin{lemma} \label{lemma47} Let $\gb_n(I)$ be the set given by Definition~\ref{def1}. Then, $\gb_n(I)$ forms a CVT with $n$-means for each $n\geq 2$ if $0<r\leq \frac {5-\sqrt {17}}{2}$, i.e., if $0<r\leq 0.4384471872$ (written up to ten decimal places).
\end{lemma}
\begin{proof} Let $a_{11}$ be the midpoint of $J_{11}$, $a_{12}$ be the midpoint of $J_{12}$, and $a_2$ be the midpoint of $J_2$. Then, $\gb_3(\set{1})=\set{a_{11}, a_{12}, a_2}$, and it will form  a CVT if
$S_{12}(1)\leq \frac 12 (a_{12}+a_2)\leq S_{2}(0)$, which implies $r\leq \frac{1}{2} \left(-r^2+r+2\right)\leq 1-r$, which after simplification yields $0<r\leq \frac{5-\sqrt{17}}{2}$, i.e., $0<r\leq 0.4384471872$. Thus, by Lemma~\ref{lemma41}, it can be seen that if $0<r\leq 0.4384471872$, $\gb_n(I)$ for each $n\geq 2$ also forms a CVT, and thus the lemma is yielded.
\end{proof}

Let us now prove the following proposition.

\begin{prop} \label{prop4} Let $\ga_n(I)$ be the set as defined by Proposition~\ref{prop1}, Proposition~\ref{prop2}, or Proposition~\ref{prop3}, and $\gb_n(I)$ be the set given by Definition~\ref{def1}. Suppose $n$ is not of the form $2^{\ell(n)}$ for any positive integer $\ell(n)$. Then,
$V(P, \ga_n(I))<V(P, \gb_n(I))$ if $0.4371985206<r\leq 0.4384471872$, where $V(P, \ga_n(I))$ and $V(P, \gb_n(I))$ respectively denote the distortion errors for the sets $\ga_n(I)$ and $\gb_n(I)$.
\end{prop}

\begin{proof} If $n$ is of the form $2^{\ell(n)}$ for some positive integer $\ell(n)$, then it is easy to see that $V(P, \ga_n(I))=V(P, \gb_n(I))$. Let us assume that $n$ is not of the form $2^{\ell(n)}$ for any positive integer $\ell(n)$. To prove $V(P, \ga_n(I))<V(P, \gb_n(I))$, due to Lemma~\ref{lemma41},
it is enough to prove the inequality for $n=3$, then it will be satisfied for all other values $n=5, 6, 7, 9, 10$, etc., which are not of the form $2^{\ell(n)}$. Notice  that in $\ga_3(I)$ as defined in Proposition~\ref{prop2}, the set $I$ is an empty set. By Lemma~\ref{lemma43}, take $\ga_3(I)=\set{a(11, 121, 1221), a(1222, 21), a(22)}$. Then, using \eqref{eq0} and \eqref{eq234}, we have
\begin{align*}
&V(P, \ga_3(I))\\
&=V(P, \set{a(11, 121, 1221), a(1222, 21), a(22)})\\
&=\int_{J_{11}\uu J_{121}\uu J_{1221}} (x-a(11, 121, 1221))^2 dP(x)  +\int_{J_{1222} \uu J_{21}} (x-a(1222, 21))^2 dP(x)\\
&\qquad \qquad \qquad+ \int_{J_{ 22}} (x-a( 22))^2 dP(x)\\
&=\int_{J_{11}} (x-a(11, 121, 1221))^2 dP(x)+\int_{J_{121}} (x-a(11, 121, 1221))^2 dP(x)\\
&\qquad \qquad \qquad+\int_{J_{1221}} (x-a(11, 121, 1221))^2 dP(x)+\int_{J_{1222}} (x-a(1222, 21))^2 dP(x)\\
&\qquad \qquad \qquad+\int_{J_{21}} (x-a(1222, 21))^2 dP(x)+\int_{J_{ 22}} (x-a( 22))^2 dP(x),
\end{align*} which implies
\begin{align*}
V(P, \ga_3(I))&=\frac 1 {2^2}\left(r^4 V+(a(11)-a(11, 121, 1221))^2\right)+\frac{1}{2^3}\left(r^6 V +(a(121)-a(11, 121, 1221))^2\right)\\
&+\frac{1}{2^4}\left(r^8 V +(a(1221)-a(11, 121, 1221))^2\right)+\frac{1}{2^4}\left(r^8 V +(a(1222)-a(1222, 21))^2\right)\\
&+\frac{1}{2^2}\left(r^4 V +(a(21)-a(1222, 21))^2\right)+\frac{1}{2^2} r^4 V.
\end{align*}
Now, use \eqref{eq45}, and simplify to obtain
\begin{equation} \label{eq51} V(P, \ga_3(I))=\frac{-3 r^9-3 r^8+14 r^7-22 r^6-71 r^5+49 r^4+4 r^3+88 r^2-84 r+28}{560 (r+1)}.
\end{equation}
To calculate $V(P, \gb_3(I))$ we take $I=\set{1}$, then $\gb_3(I)=\set{a(11), a(12), a(2)}$. Thus,
\begin{align*}
V(P, \gb_3(I))&=\int_{J_{11}}(x-a(11))^2 dP+\int_{J_{12}}(x-a(12))^2 dP+\int_{J_{2}}(x-a(2))^2 dP=\frac 1{2} r^4 V+\frac 1{2} r^2 V.
\end{align*}
We see that $V(P, \ga_3(I))<V(P, \gb_3(I))$ if $0.4371985206<r$. Combining this with the values of $r$ in Lemma~\ref{lemma47}, we see that $V(P, \ga_3(I))<V(P, \gb_3(I))$ if $0.4371985206<r\leq 0.4384471872$, which yields the proposition.
\end{proof}

\begin{remark}
 Proposition~\ref{prop4} says that if $0.4371985206<r\leq 0.4384471872$ and $n$ is not of the form $2^{\ell(n)}$ for any positive integer $\ell(n)$, then the distortion error for the CVT $\ga_n(I)$ obtained in this paper is less than the distortion error for the CVT obtained using the formula in \cite{GL2}. But, until now it is not known whether this $\ga_n(I)$ forms an optimal CVT with $n$-means for $0.4371985206<r\leq 0.4384471872$. In the following section, in  Theorem~\ref{th1}, we give an answer of it.
 \end{remark}

\section{The CVT $\ga_n(I)$ does not form an optimal CVT}

In this section, we show that if $n$ is not of the form $2^{\ell(n)}$ for any positive integer $\ell(n)$, then the CVT $\ga_n(I)$ does not form an optimal CVT for $0.4371985206<r\leq \frac{5-\sqrt{17}}{2}\approx 0.4384471872$.

Write
$
C_\gs:=\set{a(\gs11, \gs1211, \gs12121), a(\gs12122, \gs122, \gs211), a(\gs212, \gs22)}, \te{ or } \\
C_\gs:=\set{a(\gs11, \gs121), a(\gs122, \gs211, \gs21211), a(\gs21212, \gs2122, \gs22)}$. If $\gs$ is the empty word $\es$, then we have $C:=C_\es=\set{a(11, 1211, 12121), a(12122, 122, 211), a(212, 22)}$, or \\$C:=C_\es=\set{a(11, 121), a(122, 211, 21211), a(21212, 2122, 22)}$.
Let $n \in \D N$. If $n=2^{\ell(n)}+k$, where $1\leq k\leq  2^{\ell(n)-1}$, then write
\[\gd_n(I)=\left\{\begin{array}{ll}
\uu_{\gs \in I} C_\gs \uu\Big\{S_{\gs1}(\frac 1 2 ), S_{\gs2}(\frac 1 2) : \gs \in \set{1, 2}^{\ell(n)-1}\setminus I\Big\} \te{ if } 1\leq k< 2^{\ell(n)-1}, & \\
\uu_{\gs \in I} C_\gs  \te{ if } k= 2^{\ell(n)-1}\neq 1, &\\
C_\es \te{ if } k= 2^{\ell(n)-1}=1, \te{ i.e., when } n=3, &
\end{array}\right.\]
where $I \ci \set{1, 2}^{\ell(n)-1}$ with card$(I)=k=n-2^{\ell(n)}$.  If $n=3\cdot 2^{\ell(n)-1}+k$, where $1\leq k\leq  2^{\ell(n)-1}-1$, then write
\[\gd_n(I)=(\uu_{\gs \in \set{1, 2}^{\ell(n)-1}\setminus I} C_\gs)\uu \Big\{S_{\gs \gt}\Big(\frac 1 2\Big ):   \gs \in I \te{ and } \gt \in \set{ 1, 2}^2\Big\},\]
where $I \ci \set{1, 2}^{\ell(n)-1}$ with card$(I)=k=n-3\cdot 2^{\ell(n)-1}$.

We now prove the following proposition.
 \begin{prop} \label{prop15} Let  $0.4364590141\leq r\leq 0.4486234903$. Then, both $\gd_n(I)$ and $\ga_n(I)$ form CVTs.  Moreover, if $n$ is not of the form $2^{\ell(n)}$ for any positive integer $\ell(n)$, then
$V(P, \gd_n(I))<V(P, \ga_n(I))$ if  $0.4364590141\leq r\leq 0.4486234903$, where $V(P, \gd_n(I))$ and $V(P, \ga_n(I))$ respectively denote the distortion errors for the CVTs $\gd_n(I)$ and $\ga_n(I)$.
\end{prop}
\begin{proof}Let us first find the values of $r$ for which $\gd_n(I)$ forms a CVT.  Proceeding in the similar way as Lemma~\ref{lemma456}, we see that $\gd_n(I)$ forms a CVT if $0.4332840530\leq r\leq 0.4486234903$.  Moreover, by Lemma~\ref{lemma456}, $\ga_n(I)$ forms a CVT if $0.4364590141\leq r\leq 0.4512271429$. Thus,  both $\gd_n(I)$ and $\ga_n(I)$ forms a CVT if  $0.4364590141\leq r \leq 0.4486234903$. Now, to find the values of $r$ for which $V(P, \gd_n(I))<V(P, \ga_n(I))$, we proceed in the similar way as the proof of Proposition~\ref{prop4}. Take $\gd_3(I)=\set{a(11, 1211, 12121), a(12122, 122, 211), a(212, 22)}$. Then, using \eqref{eq0} and \eqref{eq234}, we have
\begin{align*}
&V(P, \gd_3(I))\\
&=\frac 1 {2^2}\left(r^4 V+(a(11)-a(11, 1211, 12121))^2\right)+\frac 1 {2^4}\left(r^8 V+(a(1211)-a(11, 1211, 12121))^2\right)\\
&+\frac 1 {2^5}\left(r^{10} V+(a(12121)-a(11, 1211, 12121))^2\right)+\frac 1 {2^5}\left(r^{10} V+(a(12122)-a(12122, 122, 211))^2\right)\\
&+\frac 1 {2^3}\left(r^{6} V+(a(122)-a(12122, 122, 211))^2\right)+\frac 1 {2^3}\left(r^{6} V+(a(211)-a(12122, 122, 211))^2\right)\\
&+\frac 1 {2^3}\left(r^{6} V+(a(212)-a(212, 22))^2\right)+\frac 1 {2^2}\left(r^{4} V+(a(22)-a(212, 22))^2\right).
\end{align*}
Then, using \eqref{eq45},
\begin{align*} \label{eq52} & V(P, \gd_3(I))\\
&=-\frac{5 r^{11}+5 r^{10}-2 r^9+18 r^8+89 r^7+21 r^6+180 r^5-48 r^4-140 r^3-568 r^2+660 r-220}{3168 (r+1)}.\end{align*}
Equation \eqref{eq51} gives $V_3(P, \ga_3(I))$. Thus, we see that $V(P, \gd_3(I))<V(P, \ga_3(I))$ if $0.4307442489<r$. Combining this with the values of $r$ for which both $\gd_3(I)$ and $\ga_3(I)$ simultaneously form a CVT, we see that $V(P, \gg_3(I))<V(P, \ga_3(I))$ if  $0.4364590141\leq r\leq0.4486234903$, which yields the proposition.
\end{proof}

Let us now give the following theorem.
\begin{theorem}\label{th1}
Let $n\in \D N$ be such that $n$ is not of the form $2^{\ell(n)}$ for any $\ell(n)\in \D N$. Let $\ga_n(I)$ be the set as defined in Section~3. Then, $\ga_n(I)$ does not form an optimal CVT for $0.4371985206<r\leq 0.4384471872$.
\end{theorem}
\begin{proof} 
By Proposition~\ref{prop15} we see that for $0.4364590141\leq r\leq 0.4486234903$, both $\gd_n(I)$ and $\ga_n(I)$ form CVTs, and $V(P, \gd_n(I))<V(P, \ga_n(I))$. Thus, $\ga_n(I)$ does not form an optimal CVT for $0.4371985206<r\leq 0.4384471872$, which is the theorem. \end{proof}

\begin{remark} \label{remark00}
Comparing Proposition~\ref{prop4} and Proposition~\ref{prop15}, if $n$ is not of the form $2^{\ell(n)}$ for any positive integer $\ell(n)$, we can say that if  $0.4364590141\leq r\leq \frac{5-\sqrt{17}}{2}\approx 0.4384471872$, then the CVT $\gb_n(I)$, which is obtained using the formula given in \cite{GL2} does not form an optimal CVT. The least upper bound of $r$ for which $\gb_n(I)$ forms an optimal CVT is still unknown. The investigation of it will appear elsewhere.
\end{remark}

\end{document}